\documentclass{article}

\usepackage{amsmath, amsfonts, amssymb, amsthm, times} 
\usepackage[numbers]{natbib}
\usepackage{srcltx}
\usepackage{setspace}

\newcommand{\diam}{\text{\,\rm {diam}}}   
\newcommand{\dd}{\text{\,\rm d}}             
\newcommand{\ee}{\text{\rm e}}

\newcommand{\Rd}{\ensuremath{\mathbf{R}^{d}}}

\newcommand{\cov}{\mathrm{cov}}

\theoremstyle{plain}
\newtheorem{theorem}{Theorem}[section]
\newtheorem{lemma}[theorem]{Lemma}
\newtheorem{defi}{Definition}[section]
\newtheorem{remark}[theorem]{Remark}

\newcommand{\ve}{{\varepsilon}}

\newcommand{\R}{{\mathbf R}}

\newcommand{\E}{{\mathbf E}}
\newcommand{\F}{{\cal F}}

\renewcommand{\P}{{\mathbf P}}
\newcommand{\X}{{\mathbb X}}

\DeclareMathOperator*{\esssup}{esssup}

\title{\bf EXPONENTIAL GROWTH RATE FOR DERIVATIVES OF STOCHASTIC FLOWS } 

\date{}
\author{HOLGER VAN BARGEN$^\ast$, MICHAEL SCHEUTZOW$^\dagger$ \and SIMON WASSERROTH$^\ddagger$\\
 {\small Institut f\" ur Mathematik, MA 7-5, Technische Universit\"at Berlin} \\ {\small Stra\ss e des 17. Juni 136,
10623 Berlin, Germany} \\ {\small $^\ast$van\_barg@math.tu-berlin.de}\\ {\small $^\dagger$ms@math.tu-berlin.de}\\{\small $^\ddagger$wasserroth@math.tu-berlin.de}}

\begin{document} 

\maketitle

{



\small

\noindent{\slshape\bfseries Abstract.} 
We show that for a large class of stochastic flows the spatial derivative  
grows at most exponentially fast even if one takes the supremum over a bounded set of initial points. We derive explicit bounds 
on the growth rates that  depend on the local characteristics of the flow and the box dimension of the set. 
\bigskip


\bigskip



}
\section{Introduction and Set-Up}
The evolution of the diameter of a bounded set under the action of a stochastic flow has been studied since the 1990's  
(see~\cite{CSS99},~\cite{CSS00},~\cite{DKK04},~\cite{LS01},~\cite{LS03}, \cite{SS02},  and the survey article~\cite{SCH09} to name just 
a few references). For a large class of flows -- including isotropic Brownian flows (IBFs)  with non-negative top Lyapunov exponent -- 
the diameter is known to be linearly growing in time. 
In this paper, we will consider the evolution of the spatial derivative of a flow in time and derive an explicit upper 
bound on the supremum of the spatial derivative taken over a bounded set.  Our bound depends on the box dimension of the set.
In the case of IBFs such a result has been obtained in~\cite{vB09} with a different (and more technical) proof. We will be much more 
general with our set-up but, contrary to  \cite{vB09}, will not derive lower bounds for the growth rates. 

We point out that there is a close link between this paper and work of Peter Imkeller: in~\cite{ImSch99}, the growth of the spatial 
derivative of a flow in the spatial direction was studied over a fixed time horizon $[0,T]$ and we did not care about constants 
(even $T$ was regarded as a constant). 
Still, the proof of Lemma \ref{le:twopointderivative}, which constitutes the core of our results, 
largely follows that of Proposition  2.3 of \cite{ImSch99}. Apart from keeping track of constants, our proof here differs from that 
in~\cite{ImSch99} towards the end 
when we apply a non-linear Gronwall-type Lemma (the usual Gronwall Lemma will not provide an exponential growth rate in $T$). 

Exponential bounds on the growth of spatial derivatives play an important role in the proof of Pesin's formula for stochastic 
flows (see~\cite{LQ95}). They can also be used to obtain bounds on the exponential growth rate of e.g.~the length of a curve under a flow. 
Even though we try to keep track of constants, we make no claims about  optimality (and we conjecture that our bound is far from optimal). 
We give explicit formulas for the exponential growth rate only for the first order derivative but indicate how such bounds 
can be obtained also for higher order derivatives under additional smoothness assumptions.

The paper is organized as follows: we start by defining a suitable class of stochastic flows. Then, we provide a general result -- 
Theorem \ref{th:zweites} -- which shows how one can obtain exponential growth rates for a random field $\psi$ indexed by $\R^d$ given moment 
bounds on the field and on two-point differences of the field. Afterwards, we apply this theorem to the derivative of a stochastic flow. 
Here, the main task is to compute the moment bounds needed in order to apply Theorem \ref{th:zweites}. Then, we specialize to IBFs.\\ 
  

Let us introduce our set-up which is essentially the same as in  \cite{ImSch99} and is based on \cite{Ku90}.

Let $F(x,t), t\geq0$ be a family of $\Rd$-valued continuous semimartingales on a filtered probability space 
$(\Omega,\F,(\F_t)_{t\geq 0},\P)$ indexed by $x\in\Rd$, starting at 0. 
Let $F(x,t)=M(x,t)+V(x,t)$ be the canonical decomposition into a local martingale $M$ and a process $V$ of locally bounded variation 
(both starting at 0). We will assume throughout that both $M$ and $V$ are jointly continuous in $(x,t)$. Furthermore we assume 
that there exist $a:\Rd\times\Rd\times[0,\infty[\times\Omega\to\R^{d\times d}$ which is continuous in the first two and predictable in the 
last two variables and $b:\Rd\times[0,\infty[\times\Omega\to\Rd$ which is continuous in the first and predictable in the 
last two variables such that
$$\langle M_i(x,.),M_j(y,.)\rangle(t)=\int_0^ta_{ij}(x,y,u)\dd u,\hspace{5mm}V_i(x,t)=\int_0^tb_i(x,u)\dd u.  $$ 
Here, $\langle.,.\rangle$ denotes the joint quadratic variation. 
The pair of random fields $(a,b)$ is called the {\em local characteristics} of the semimartingale field $F$.
We will abbreviate $\mathcal{A}(x,y,t):=a(x,x,t)-a(x,y,t)-a(y,x,t)+a(y,y,t)$ 
(which is the derivative of the quadratic variation of $M(x,t)-M(y,t)$). Throughout, we will assume that the following hypothesis 
holds:\\
\newpage
\noindent {\bf Hypothesis (A):}

\begin{align*}
\esssup_{\omega\in\Omega}&\sup_{t \ge 0} \sup_{x,y\in \Rd}\Big(\frac{\|a(x,y,t)\|}{(1+|x|)(1+|y|)}
+ \sum_{k=1}^d \|D_{x_k} D_{y_k} a(x,y,t)\|\\ &+ \sum_{k=1}^d \| D_{x_k} D_{y_k}a(.,.,t)\|^\sim \Big)<\infty, \mbox{ and }\\
\esssup_{\omega\in\Omega}&\sup_{t \ge 0} \Big( \sup_{x\in \Rd} \frac{|b(x,t)|}{1+|x|}+\sup_{x\in\Rd}\|D_x b(x,t)\|+
\sup_{x\neq y \in\Rd}\frac{ \|D_xb(x,t)-D_y b(y,t)\|}{|x-y|}\Big)<\infty,
\end{align*}
where 
$$
 \| f(.,.)\|^\sim:=\sup_{x\neq x',y\neq y'}\left\{\frac{\|f(x,y)-f(x',y)-f(x,y')+f(x',y')\|}{|x-x'|\,|y-y'|} \right\}.
$$
Since Hypothesis (A) implies the assumptions of \cite[Theorem 4.6.5]{Ku90} (with $k=1$, $\delta=1$), the stochastic differential equation
\begin{equation}\label{sde} \dd X(t)=F(X(t),\dd t ),\qquad X(s)=x,\;t\ge s \end{equation} 
not only admits a unique solution for each fixed $x \in \R^d$ and $s \ge 0$, but even generates a 
{\em stochastic flow of diffeomorphisms} $\phi$, i.e.~there exist a random field 
$\phi:[0,\infty)^2 \times \R^d \times \Omega \to \R^d$ and a set $\Omega_0$ of full measure such that
\begin{itemize}
\item $t \mapsto \phi_{s,t}(x)$, $t\ge s$ solves \eqref{sde} for all $x \in \R^d$, $s \ge 0$.
\item $\phi_{s,t}(\omega)$ is a diffeomorphism on $\Rd$ for all $s,t \ge 0$, $\omega \in \Omega_0$.
\item $\phi_{s,u}=\phi_{t,u} \circ \phi_{s,t}$ for all $s,t,u \ge 0$, $\omega \in \Omega_0$.
\item $(s,t,u) \mapsto \phi_{s,t}(x)$ is continuous for all  $\omega \in \Omega_0$. 
\end{itemize}
 
We will often write $x_t:=\phi_t(x):=\phi_{0,t}(x,\omega) $. 
 


\section{Exponential Growth Rates: General Results}


In the following lemma and theorem, $o(T)$ stands for a function $g(T)$ which may depend on $q$, but not on $x,y$ 
such that $\lim_{T \to \infty} g(T)/T=0$.

\begin{lemma}\label{erstes} Let $(E,\rho)$ be a complete, separable metric space and let 
$(t,x) \mapsto \psi_t(x)$ be a continuous 
random field on $[0,\infty) \times \R^d$ with values in $(E,\rho)$ which satisfies
$$
\E \sup_{0 \le t \le T} \rho(\psi_t(x),\psi_t(y))^q \le |x-y|^q \exp\{ (cq^2+\hat c q)T + o(T)\}
$$
for some $q > d$ and all $x,y \in \R^d$.
Then, for $u>0$, we have
\begin{align*}
\P \Big\{ &\sup_{x,y \in [0,1]^d} \sup_{0 \le t \le T}  \rho(\psi_t(x),\psi_t(y)) \ge u \Big\}
\le \exp\{(cq^2+\hat c q)T + o(T)\} u^{-q},
\end{align*}
\end{lemma}
\begin{proof} This follows from Kolmogorov's continuity theorem (e.g.~in the version of~\cite[Lemma 2.1]{SCH09}).
\end{proof}


\begin{theorem}\label{th:zweites} Let $(E,\|.\|)$ be a separable real Banach space and let $(t,x) \mapsto \psi_t(x)$ 
be a continuous random field on $[0,\infty) \times \R^d$ with values in $(E,\|.\|)$ which satisfies
\begin{equation}\label{oans}
\E \sup_{0 \le t \le T} \|\psi_t(x)-\psi_t(y)\|^q \le |x-y|^q \exp\{ (cq^2+\hat c q)T + o(T)\}
\end{equation}
for some $c>0$, $\hat c \in \R$, all $q > d$ and all $x,y \in \R^d$. 
Assume further that 
\begin{equation}\label{zwo}
\sup_{x}\E \sup_{0 \le t \le T} \|\psi_t(x)\|^q \le \exp\{ (kq^2+\hat k q)T + o(T)\}
\end{equation}
for some $k>0$, $\hat k \in \R$ and for all $q \ge 0$. If $\X$ is any compact subset of $\R^d$ with box 
dimension $\Delta$, then
\begin{equation}\label{toshow}
\limsup_{T \to \infty} \frac 1T \log \sup_{x \in \X} \sup_{0 \le t \le T}  \|\psi_t(x)\| \le \xi \;\mbox{ a.s.},
\end{equation}
where $$\xi=
\left\{\begin{array}{ll}
\hat k & \mbox{ if }\; \hat k \ge cd+\hat c\\
\hat k + 2\sqrt{k\Delta \gamma_1}&\mbox{ if }\;  \hat k \le cd+\hat c \;\mbox{ and }\; 
2\sqrt{ck}\Delta d+cd^2-2c\Delta d +\Delta(\hat k-\hat c)\ge 0\\
\hat k + 2\sqrt{k\Delta \gamma_2} &\mbox{ if }\;  \hat k \le cd+\hat c \;\mbox{ and } \;
2\sqrt{ck}\Delta d+cd^2-2c\Delta d +\Delta(\hat k-\hat c)\le 0,
\end{array}\right.
$$
where  
\begin{align*}
\gamma_1:&=
\left\{\begin{array}{ll}\frac{cd + \hat c - \hat k}{2\sqrt{kd}}& \mbox{ if }\;\Delta = d\\
\Big( \Big(1-\frac {\Delta}{d}\Big)^{-1} \Big(-\sqrt{k\Delta}
+\sqrt{k\Delta+\big(1-\frac{\Delta}d\big)(cd+\hat c-\hat k)}\Big) \Big)^2& \mbox{ if }\;\Delta < d
\end{array}\right.\\
\gamma_2:&=\Big(\sqrt{c\Delta}+\sqrt{k\Delta}+\sqrt{(\sqrt{c}-\sqrt{k})^2 \Delta - \hat k + \hat c}\Big)^2.
\end{align*}
\end{theorem}
\begin{proof}
Let $\ve >0$. For each $\gamma >0$, we can cover the set $\X$ with $N\le \ee^{\gamma T(\Delta + \ve)}$ balls of diameter 
$\ee^{-\gamma T}$ in case $T$ is large enough. For given such $\gamma,T$ we denote these balls by $\X_1,...\X_N$ and 
their centers by $x_1,...,x_N$. Let $r >0$.
Then, using Lemma~\ref{erstes}, we obtain
$$
\P\big\{\exists i \in \{1,...,N\} \mbox{ s.t. } \sup_{0 \le t \le T} \diam (\psi_t(\X_i)) \ge \ee^{rT}\big\}
\le \ee^{\gamma T(\Delta + \ve)}\ee^{(cq^2+(\hat c - \gamma -r)q)T + o(T)},
$$
and therefore
\begin{align*}
B(r):=&\limsup_{T \to \infty} \frac 1T \log \P\big\{\exists i \in \{1,...,N\} 
\mbox{ s.t. } \sup_{0 \le t \le T} \diam (\psi_t(\X_i)) \ge \ee^{rT}\big\}\\ \le&\gamma \Delta + cq^2+(\hat c-\gamma -r)q.
\end{align*}
Optimizing over $q > d$ yields
$$
B(r) \le \left\{
\begin{array}{cc}
\gamma\Delta-\frac{(r-\hat c+\gamma)^2}{4c}&\mbox {if } r\ge 2cd+\hat c-\gamma\\
\gamma\Delta+(cd+\hat c-\gamma-r)d&\mbox {if } r< 2cd+\hat c-\gamma\end{array}
 \right..
$$
Further,
$$
\P\big\{\exists i \in \{1,...,N\} \mbox{ s.t. } \sup_{0 \le t \le T}\|\psi_t(x_i)\|\ge  \ee^{rT}\big\} 
\le \ee^{\gamma T(\Delta + \ve)}\ee^{(kq^2+(\hat k  -r)q + o(1))T}
$$
and therefore
$$
C(r):=\limsup_{T \to \infty} \frac 1T \log \P\big\{\exists i \in \{1,...,N\} \mbox{ s.t. } 
\sup_{0 \le t \le T}\|\psi_t(x_i)\|\ge  \ee^{rT}\big\} \le \gamma \Delta + kq^2 + (\hat k - r)q.
$$
Optimizing over $q \ge 0$, we get
$$
C(r) \le \gamma \Delta - \frac{(r-\hat k)^2}{4k}
$$ 
provided that $r \ge \hat k$ which we will assume to hold from now on.

Once we know that for a particular value of $r>0$
$$
A(r):=\limsup_{T \to \infty} \frac 1T \log \P\{ \sup_{0\le t\le T}\sup_{x \in \X} \|\psi_t(x)\|\ge 2e^{rT}\} <0,
$$
then a simple Borel-Cantelli argument (using the fact that 
$T \mapsto \sup_{0\le t\le T}\sup_{x \in \X} \|\psi_t(x)\|$ 
is non-decreasing) shows that \eqref{toshow} holds with $\xi$ replaced by $r$. Since 
$A(r) \le B(r) \vee C(r)$, we have $A(r)<0$ whenever there exists some $\gamma>0$ such that both upper bounds of 
$B(r)$ and $C(r)$ are negative. 
Defining $\xi$ as  the infimum over all such $r$, we obtain
\begin{equation}\label{ximax}
\xi=\inf_{\gamma>0}\Big(
(\hat k +2\sqrt{k\gamma\Delta}) \vee \left\{\begin{array}{cc}
2\sqrt{c\gamma\Delta}+\hat c-\gamma&:\gamma \Delta \ge c d^2\\
\gamma\Delta d^{-1}+cd+\hat c-\gamma &:\gamma \Delta \le c d^2
\end{array}\right.\Big).
\end{equation}
Computing the infimum, we obtain the result in the theorem.
\end{proof}
\begin{remark} It follows from \eqref{ximax} that 
$$
\hat k \le \xi \le (cd + \hat c)\vee \hat k.
$$
Note that the lower bound is attained in case $\Delta=0$.
\end{remark}
\begin{remark}
Our assumptions on the range of admissible values of $q$ in Theorem \ref{th:zweites} are a bit arbitrary (but motivated by 
applications to IBFs). It is clear from the proof that  \eqref{toshow} still holds with a larger value of $\xi$ if 
we only assume that \eqref{oans} holds for {\em some} $q>d$ and that  \eqref{zwo} holds for {\em some} $q>0$ (the values of 
$q$ can be different).
\end{remark}

%
%
%
%
%
\section{Application to the Derivative of a Stochastic Flow}
Next, we want to use the results in the previous section to obtain bounds on the exponential growth rate of 
the supremum of  the  derivative of a stochastic flow taken over all initial points in a compact set 
of box dimension $\Delta$. In order to apply Theorem \ref{th:zweites}, we have to estimate moments of the difference 
of derivatives of a stochastic flow. We start by introducing some more notation (as in~\cite{ImSch99}).\\

Let $\phi=(\phi^1,\ldots,\phi^d)$ be a stochastic flow of diffeomorphisms generated by~\eqref{sde} 
satisfying Hypothesis (A) and let $p\geq1$. Fix $i\in\{1,\ldots,d\}$ and define
$Y_j(t):=D_i\phi^j_t(x)$, $Z(t):=\sup_{0\leq s\leq t}\big(\sum_{j=1}^d Y_j^2(s)\big)^{1/2}$, 
$f_p(t)=\max_i\big(\E Z(t)^p \big)^{1/p} $, and $\tilde f_p(t)=\max_i\big(\E\big(\sum_{j=1}^d Y_j^2(t) \big)^{p/2}\big)^{1/p}$. 
Here and in the following, we write $D_k$ instead of $D_{x_k}$ or $D_{y_k}$. 
We will also need two two-point versions   
$V_j(t):=D_i \phi^j_{0,t}(x) - D_i \phi^j_{0,t}(y)$, $W(t):=\sup_{0\le s\le t} \big(\sum_{j=1}^d  V_j^2(s) \big)^{1/2} $ 
and $g_p(t):=\max_i\big(\E\big(W^p(t)\big)^{1/p}$. Let $C_p$ denote the constant in Burkholder's inequality. It is well-known, that 
there exists a constant $k_5>0$ such that $C_p \le (k_5 p^{1/2})^p$ for all $p \ge 2$ (see \cite{BY82}) (one can choose $k_5=2\sqrt{5}$).
We point out that $D_k D_k$ in front of a function of two spatial arguments means that we differentiate with respect to the $k-$th 
component of {\em both} arguments. Note that we trivially have $\tilde f_p(t) \le f_p(t)$, so the terms $\tilde f_p(t)$ 
which appear in the upper bound of $H$ in the following lemma can be replaced by $f_p(t)$ (but one may get better bounds by not doing this). 


\begin{lemma}\label{le:twopointderivative} 
Let $\phi$ be the flow generated by $F$ satisfying Hypothesis (A) and denote 
\begin{enumerate}
 \item $k_1:=\mathrm{esssup}_{\omega \in \Omega}\sup_{\bar x\in\Rd,t\geq 0,1\le j,k\le d}|D_kD_ka_{jj}(\bar x,\bar x,t)|$,
 \item $k_2:=\mathrm{esssup}_{\omega \in \Omega} \sup_{\bar x,\bar y\in\mathbf{R}^d,t\geq0,1\le j,k \le d} 
\frac{|D_k D_k \mathcal{A}_{jj}(\bar x,\bar y,t)|}{|\bar x-\bar y|^2}$, 
 \item $k_3:=\mathrm{esssup}_{\omega \in \Omega}\sup_{\bar x\in\Rd,t\geq 0,1\le j,k\le d}|D_kb_j(\bar x,t)|$, 
 \item $k_4:=\mathrm{esssup}_{\omega \in \Omega}\sup_{\bar x,\bar y\in\mathbf{R}^d,t\geq0,1\le j,k\le d} 
\frac{ |D_k b_j(\bar x,t)-  D_k b_j(\bar y,t)|}{|\bar x-\bar y|}$. 
\end{enumerate}
Then there exist constants $\Lambda$, $\bar{c}$ and $\sigma \ge 0$ such that for all $p\geq2$, $t \ge0$ and $\bar x,\bar y \in \R^d$, 
\begin{equation}\label{esti}
\Big(\E |\phi_t(\bar x)-\phi_t(\bar y)|^p\Big)^{1/p}\le \bar{c}|\bar x-\bar y|\ee^{(\Lambda +p\sigma^2/2)t}. 
\end{equation}
Further, for all $p \ge 2$, $x \in \R^d$, and all $\alpha_1,\,\alpha_2,\,\alpha_3>1$ whose reciprocals sum up to 1, 
$f_p$ satisfies
\begin{equation}\label{eq:fgleichung}
f_p^2(t)\leq \alpha_1 + \alpha_1 \frac{\alpha_2\bar{d}^2k_1C_p^{2/p}+\sqrt{\alpha_3}\bar{d}k_3}{\alpha_2\bar{d}^2k_1C_p^{2/p}+2\sqrt{\alpha_3}\bar{d}k_3} 
\Big(\exp\{(\alpha_2\bar{d}^2k_1C_p^{2/p}+2\sqrt{\alpha_3}\bar{d}k_3)t\}-1\Big),
\end{equation}
where $\bar{d} = d^{2-1/p}$.
Further, for all $p \ge 2$, $x,y \in \R^d$, and all $\beta_m>1$, $m=1,2,3,4$ whose reciprocals sum up to 1, we have
\begin{equation}\label{eq:ggleichung}
g_p^2(t)\leq H(t) + H(t)\frac{\bar C_1+\sqrt{\bar C_2}}{\bar C_1+2\sqrt{\bar C_2}} \Big(\exp\{(\bar C_1+2\sqrt{\bar C_2})t\}-1\Big),
\end{equation}
where $\bar C_1:=\beta_1d^{3-2/p} C_p^{2/p} k_1$, $\bar C_2:=\beta_3d^{3-2/p}k_3^2 $ and
$$
H(t):=d^3\bar{c}^2|x-y|^2\Big(C_p^{2/p}\beta_2 k_2 \int_0^t \tilde f^2_{2p}(s)\ee^{2(\Lambda+p\sigma^2)s}  \dd s  
+ \beta_4 k_4^2\Big(\int_0^t\tilde f_{2p}(s)\ee^{(\Lambda+\sigma^2p)s}\dd s \Big)^2\Big).
$$
\end{lemma}



\begin{proof}
First note that $k_1,...,k_4<\infty$ since $F$ satisfies Hypothesis (A). Assertion \eqref{esti} follows from \cite[Lemma 2.6]{SCH09} 
and the fact that $F$ satisfies (A) (one can choose $\sigma=\tilde a$ and $\Lambda=\tilde b+(d-1)\tilde a^2/2$, 
where $\tilde b$ is a deterministic upper bound of the Lipschitz constant of $b$ and $\tilde a\ge 0$ is chosen such that 
$\|\mathcal{A}(x,y,t,\omega)\|\le \tilde a^2|x-y|^2$ for all $x,y \in \R^d$ and almost all $\omega \in \Omega$). Fix $i \in \{1,...,d\}$. 
We have
$$
Y_j(s)= \delta_{i,j}+\sum_{n=1}^d \int_0^s Y_n(u)D_n F_j(x_u,\dd u)
$$
(see \cite[p.174 (21)]{Ku90} or~\cite[(18)]{ImSch99}).
Applying Burkholder's inequality, we get
\begin{align}
\Big(\E&\sup_{0\leq s\leq t}|Y_j(s)|^p \Big)^{1/p}
\le \delta_{i,j}+\sum_{n=1}^d \Big(\E\big(\sup_{0\leq s\leq t}\int_0^s Y_n(u)D_n F_j(x_u,\dd u)\big)^p  \Big)^{1/p}\nonumber\\
&\le\delta_{i,j}+ C_p^{1/p}\sum_{n=1}^d\Big(\E\big(\int_0^t Y^2_n(u)D_nD_na_{jj}(x_u,x_u,u)\dd u\big)^{p/2}\big)^{1/p}\nonumber\\
& \hspace{.5cm} + \sum_{n=1}^d\Big(\E\big(\int_0^t |Y_n(u)  D_nb_j(x_u,u)|\dd u\big)^p\Big)^{1/p}\nonumber\\
&\le\delta_{i,j}+C_p^{1/p}\sqrt{k_1} \sum_{n=1}^d \Big( \E \big(\int_0^t Y_n^2(u)\dd u\big)^{p/2}\Big)^{1/p}\nonumber\\ 
& \hspace{.5cm} + k_3 \sum_{n=1}^d\Big(\E\big(\int_0^t |Y_n(u)|\dd u\big)^p\Big)^{1/p}.\label{zweites}
\end{align}
Since $p \ge 2$, Jensen's inequality implies
\begin{align*}
\sum_{n=1}^d &\Big( \E \big(\int_0^t Y_n^2(u)\dd u\big)^{p/2}\Big)^{1/p} 
\le \sqrt{d} \Big(\sum_{n=1}^d \Big( \E \big(\int_0^t Y_n^2(u)\dd u\big)^{p/2}\Big)^{2/p}\Big)^{1/2} \\
&\le \sqrt{d} \Big(\int_0^t \sum_{n=1}^d\big(\E |Y_n(u)|^p\big)^{2/p}\dd u\Big)^{1/2}
\le \sqrt{d} \Big(\int_0^t d^{1-\frac 2p}\big(\E \sum_{n=1}^d |Y_n(u)|^p\big)^{2/p}\dd u\Big)^{1/2}
\\
&\le d^{1-1/p}\Big(\int_0^t \Big(\E \big(\sum_{n=1}^d Y_n^2(u)\big)^{p/2}\Big)^{2/p}\dd u\Big)^{1/2} \le  d^{1-1/p} \Big(\int_0^tf_p^2(u)\dd u\Big)^{1/2}.
\end{align*}
The term \eqref{zweites} can be estimated similarly:
\begin{align*}
\sum_{n=1}^d&\Big(\E\big(\int_0^t |Y_n(u)|\dd u\big)^p\Big)^{1/p} \le \int_0^t  \sum_{n=1}^d \big( \E |Y_n(u)|^p\big)^{1/p} \dd u\\ 
&\le d^{1-1/p} \int_0^t  \Big(\E \sum_{n=1}^d  |Y_n(u)|^p \Big)^{1/p} \dd u \le d^{1-1/p} 
\int_0^t \Big( \E \Big(\sum_{n=1}^d  Y_n^2(u)\Big)^{p/2}\Big)^{1/p}\dd u\\
&\le d^{1-1/p} \Big(\int_0^tf_p(u)\dd u\Big).
\end{align*}
Therefore we obtain 
\begin{align*}
f_p(t)&=\max_i (\E Z^p(t))^{1/p}\leq \max_i\sum_{j=1}^d \Big(\E\sup_{0\leq s\leq t}|Y_j(s)|^p \Big)^{1/p} \\
&\le 1+d^{2-1/p}\sqrt{k_1}C_p^{1/p}\Big(\int_0^tf_p^2(s)\dd s\Big)^{1/2}+d^{2-1/p}k_3\int_0^tf_p(s) \dd s. 
\end{align*}
Taking squares and using the formula $(A+B+C)^2 \le \alpha_1 A^2 +  \alpha_2 B^2 +  \alpha_3 C^2$ for $A,B,C \ge 0$, we obtain 
\begin{align*}
f^2_p(t)\leq \alpha_1+\alpha_2 d^{4-2/p}k_1C_p^{2/p}\int_0^tf_p^2(s)\dd s+\alpha_3 d^{4-2/p}k_3^2\big(\int_0^t f_p(s) \dd s\big)^2 ,
\end{align*}
and hence~\eqref{eq:fgleichung} by Lemma~\ref{le:gronwall}. 

Let us now treat the two-point differences and recall from~\cite[p. 123]{ImSch99} that
$$
V_j(t)=\sum_{n=1}^d \Big( \int_0^t V_n(s) D_n F_j(x_s,\dd s) 
+ \int_0^t D_i\phi_s^n(y)\big( D_n F_j(x_s,\dd s)-  D_n F_j(y_s,\dd s)\big) \Big).
$$ 
Then for $p \ge 2$ we have
\begin{align*}
\Big( \E&\big( \sum_{j=1}^d \sup_{0 \le t \le T} V_j^2(t) \big)^{p/2} \Big)^{2/p}
\le  \sum_{j=1}^d \big( \E \sup_{0 \le t \le T} V_j^p(t)\big)^{2/p}\\ 
&\le  \sum_{j=1}^d \Big( \sum_{n=1}^d \Big(\E \sup_{0 \le t \le T} \Big| \int_0^t V_n(s) D_n M_j(x_s,\dd s)    \Big|^p   \Big)^{1/p}\\ 
&\hspace{.4cm}+ \sum_{n=1}^d \Big(\E \sup_{0 \le t \le T}  \Big| \int_0^t  D_i\phi_s^n(y)
\big( D_n M_j(x_s,\dd s)-  D_n M_j(y_s,\dd s)\big)      \Big|^p     \Big)^{1/p}  \\
&\hspace{.4cm}+ \sum_{n=1}^d \Big(\E \sup_{0 \le t \le T} \Big| \int_0^t V_n(s) D_n b_j(x_s,s) \dd s    \Big|^p   \Big)^{1/p}\\ 
&\hspace{.4cm}+ \sum_{n=1}^d\Big( \E \sup_{0 \le t \le T}  \Big| \int_0^t  D_i\phi_s^n(y)
\big( D_n b_j(x_s,s)\dd s-  D_n b_j(y_s,s)\big) \dd s\big)      \Big|^p     \Big)^{1/p}   \Big)^2.
\end{align*}
We have by Burkholder's inequality
\begin{align*}
\E \sup_{0 \le t \le T} \Big| \int_0^t V_n(s) D_n M_j(x_s,\dd s)    \Big|^p  
\le& C_p \E \Big( \Big| \int_0^T V_n^2(t) D_n D_n a_{jj}(x_t,x_t,t) \dd t\Big|^{p/2}\Big)\\
\le& C_p k_1^{p/2} \E \Big( \Big| \int_0^T V_n^2(t) \dd t\Big|^{p/2}\Big)
\end{align*}
and
\begin{align*}
\Big(\E &\sup_{0 \le t \le T}  \Big|\int_0^t  D_i\phi_s^n(y)
\big( D_n M_j(x_s,\dd s)-  D_n M_j(y_s,\dd s)\big)      \Big|^p\Big)^{1/p}  \\
\le& C_p^{1/p} \Big(\E  \Big| \int_0^T (D_i\phi^n_s(y))^2 D_n D_n \mathcal{A}_{jj}(x_s,y_s,s) \dd s\Big|^{p/2}\Big)^{1/p}\\
\le& C_p^{1/p} k_2^{1/2} \Big(\E  \Big| \int_0^T (D_i\phi^n_s(y))^2 |\phi_s(x)-\phi_s(y) |^2 \dd s\Big|^{p/2} \Big)^{1/p}\\
\le& C_p^{1/p} k_2^{1/2} \Big| \int_0^T \big(\E| D_i\phi^n_s(y)|^p |\phi_s(x)-\phi_s(y) |^p\big)^{2/p} \dd s\Big|^{1/2} \\
\le& C_p^{1/p} k_2^{1/2} \Big| \int_0^T \big(\E| D_i\phi^n_s(y)|^{2p}\big)^{1/p} \big(\E |\phi_s(x)-\phi_s(y) |^{2p}\big)^{1/p}  \dd s\Big|^{1/2} \\
\le& C_p^{1/p} k_2^{1/2} \Big| \int_0^T \tilde f^2_{2p}(s)\bar{c}^2|x-y|^2\ee^{2(\Lambda+p\sigma^2)s}  \dd s\Big|^{1/2}
\end{align*}
and
\begin{align*}
\E \sup_{0 \le t \le T} \Big| \int_0^t V_n(s) D_n b_j(x_s,s)\dd s    \Big|^p   \le k_3^p \E  \Big( \int_0^T |V_n(s)| \dd s    \Big)^p \\
\end{align*}
and
\begin{align*}
&\Big(\E \sup_{0 \le t \le T}  \Big| \int_0^t  D_i\phi_s^n(y)
\big( D_n b_j(x_s,s)-  D_n b_j(y_s,s)\big)\dd s      \Big|^p     \Big)^{1/p}\\   
\le&k_4\Big(\E\big(\int_0^T |D_i\phi^n_s(y) | |\phi_s(x)-\phi_s(y)| \dd s \big)^p \Big)^{1/p} 
\leq k_4\bar{c}|x-y|\int_0^T\tilde f_{2p}(s)\ee^{(\Lambda+\sigma^2p)s}\dd s.
\end{align*}
Therefore, using the same estimates as in the first part of the proof, we get
\begin{align*}
g_p^2(T)\le&d\Big( d^{1-1/p}C_p^{1/p} k_1^{1/2}  \Big( \int_0^T g_p^2(t) \dd t\Big)^{1/2}\\   
& +dC_p^{1/p}\bar{c}|x-y| k_2^{1/2} \Big( \int_0^T \tilde f^2_{2p}(s)\ee^{2(\Lambda+p\sigma^2)s}  \dd s\Big)^{1/2}\\
& +d^{1-1/p}k_3   \int_0^T g_p(s) \dd s + 
dk_4\bar{c}|x-y|\int_0^T\tilde f_{2p}(s)\ee^{(\Lambda+\sigma^2p)s}\dd s \Big)^2\\
&\le   \beta_1d^{3-2/p} C_p^{2/p} k_1 \int_0^T g^2_p(t) \dd t+ \beta_3 d^{3-2/p}k_3^2 \Big(\int_0^T g_p(t) \dd t\Big)^2 +H(T),
\end{align*}
where $H$ is as in the lemma. Therefore, \eqref{eq:ggleichung} follows from Lemma \ref{le:gronwall} and the proof is complete.
\end{proof}

\begin{theorem}\label{main}
Let $d \ge 2$ and let $\phi$ be a stochastic flow satisfying the assumptions of Lemma~\ref{le:twopointderivative}. 
For any (deterministic, compact) subset $\X$ of $\Rd$ with 
box dimension $\Delta \ge 0$, we have
$$
\limsup_{T \to \infty} \frac 1T \log \sup_{x \in \X} \sup_{0 \le t \le T}  \|D\phi_{0,t}(x)\| \le \xi,
$$ 
where
 $$\xi=
\left\{\begin{array}{ll}
\hat k & \mbox{ if }\; \hat k \ge cd+\hat c\\
\hat k + 2\sqrt{k\Delta \gamma_1}&\mbox{ if }\;  \hat k \le cd+\hat c \;\mbox{ and }\; 
2\sqrt{ck}\Delta d+cd^2-2c\Delta d +\Delta(\hat k-\hat c)\ge 0\\
\hat k + 2\sqrt{k\Delta \gamma_2} &\mbox{ if }\;  \hat k \le cd+\hat c \;\mbox{ and } \;
2\sqrt{ck}\Delta d+cd^2-2c\Delta d +\Delta(\hat k-\hat c)\le 0,
\end{array}\right.
$$
where $\gamma_1$ and $\gamma_2$ are defined as in Theorem~\ref{th:zweites} 
and
\begin{align*}
k&=\alpha_2d^4k_1k_5^2/2,\qquad \hat k= \sqrt{\alpha_3} d^2 k_3 + 2k\\
c&=\alpha_2 d^4 k_1 k_5^2 + \frac 12 \beta_1 d^3 k_1 k_5^2 + \sigma^2 ,\qquad \hat c= \sqrt{\alpha_3}d^2 k_3 
+ \sqrt{\beta_3 d^3 k_3^2} +\Lambda. 
\end{align*}
\end{theorem}
\begin{proof}
The proof is just a combination of  Lemma~\ref{le:twopointderivative} and Theorem~\ref{th:zweites}. 
Lemma~\ref{le:twopointderivative} tells us that
$$
\big(\E \sup_{0 \le t \le T} \|D\phi_t(x)\|^p\big)^{1/p} \le \exp\{(kp+\tilde k)T + o(T)\}
$$
for all $p \ge 2$, where
$$
k=\alpha_2d^4k_1k_5^2/2,\qquad \tilde k= \sqrt{\alpha_3} d^2 k_3.
$$
In order to obtain an estimate for all $p \ge 0$, we define $\hat k:=2k+\tilde k$ and get
$$
\big(\E \sup_{0 \le t \le T} \|D\phi_t(x)\|^p\big)^{1/p} \le \exp\{(kp+\hat k)T + o(T)\}
$$
for all $p \ge 0$.
 
Lemma~\ref{le:twopointderivative} tells us further that for $p\ge 2$ (and hence for $p>d$)
$$
\E \sup_{0 \le t \le T} \|D\phi_t(x)-D\phi_t(y)\|^p\le 
|x-y|^p \exp\{ (cp^2+\hat c p)T+o(T)\},
$$
with $c,\bar c$ as in the theorem. Therefore, the assumptions of  Theorem~\ref{th:zweites} hold
 and the assertion follows.
\end{proof}
\begin{remark}
The formulas in Theorem \ref{main} still contain the numbers $\alpha_2,\,\alpha_3,\,\beta_1$ and $\beta_3$. Since 
$\alpha_1, \, \beta_2$ and $\beta_4$ do not appear in the formulas, it is possible to choose 
$\alpha_2=\alpha_3=\beta_1=\beta_3=2$ but a different choice may result in a sharper bound.
\end{remark}
\begin{remark}
One can establish a corresponding result also in case $d=1$ by adjusting $\hat c$ just like we adjusted $\hat k$ in order to extend the 
moment bound on the differences of derivatives from $p \ge 2$ to $p > d=1$.
\end{remark}
\begin{remark}
It is not hard to establish a version of both Lemma~\ref{le:twopointderivative} and Theorem~\ref{main} for higher derivatives of 
a stochastic flow using an induction proof along the lines of~\cite[Proposition 2.3]{ImSch99}. 
\end{remark}

\section{Isotropic Brownian Flows}\label{se:IBF}

In this section we will specialize the results of Theorem \ref{main} to isotropic Brownian flows (IBFs). We will be able to 
establish somewhat better upper bounds by exploiting -- for example -- an explicit representation of the growth 
of the derivative of an IBF. We start by defining an IBF. We assume that $d \ge 2$.

\begin{defi}
A stochastic flow $\phi$ on $\R^d$ is called {\em isotropic Brownian flow} if it is generated by a semimartingale field 
$F=M+V$ with $V \equiv 0$ and martingale field $M$ with 
quadratic variation $a(x,y,t,\omega)=b(x-y)$, where $b:\R^d \to \R^{d \times d}$ is a deterministic function satisfying
\begin{itemize} 
\item $x \mapsto b(x)$ is $C^4$.
\item $b(0)=\mathrm{id}_{\R^d}$.
\item $x \mapsto b(x)$ is not constant.
\item $b(x)=O^*b(Ox)O$ for any $x \in \R^d$ and any orthogonal matrix $O \in O(d)$.
\end{itemize}
\end{defi}

For an IBF, we define its {\em longitudinal} resp.~{\em normal} correlation functions by
$$
B_L(r):=b_{11}(r e_1),\qquad B_N(r):= b_{11}(r e_2), \quad r \ge 0,
$$ 
where $e_i$ denotes the $i^{\mathrm{th}}$ unit vector in $\R^d$ (1 and 2 can be replaced by any $i\neq j$ by isotropy).

We will need the following facts about IBFs. 
\begin{itemize}
\item $\beta_L:=-B_L''(0)>0$,  $\beta_N:=-B_N''(0)>0$.
\item For each $x \in \R^d$, we have $\lambda_1:=\frac 12 \big( (d-1)\beta_N-\beta_L \big)=\lim_{t \to \infty}\frac 1t \log \|D\phi_t(x)\|$
almost surely. The number $\lambda_1$ is called the {\em top Lyapunov exponent} of the IBF. 
\item $b,B_L$ and $B_N$ are bounded with bounded derivatives up to order 2.
\item $k_1:=\max_{i,j}|D_{x_i}D_{y_i}b_{jj}(x-y)|_{x=y}=\beta_L \vee \beta_N$ (independently of $x$)           
\item $F$ satisfies Hypothesis (A). 
\end{itemize}
The first two of these facts can be found in  \cite{BaH86}, $b$ is bounded since $b$ is a covariance function and 
boundedness of the second derivatives (and therefore also of the first) follows from equation \eqref{cov}. The fourth item 
follows from the definition of $\beta_L$ and $\beta_N$ and the final one from the boundedness of the second derivatives of $b$.

\begin{theorem}\label{iso}
Let $d \ge 2$ and let $\phi$ be an IBF.  For any (deterministic, compact) subset $\X$ of $\Rd$ with 
box dimension $\Delta \ge 0$, we have
$$
\limsup_{T \to \infty} \frac 1T \log \sup_{x \in \X} \sup_{0 \le t \le T}  \|D\phi_{0,t}(x)\| \le \xi,
$$ 
where $\xi$ is defined as in Theorem \ref{th:zweites} with
$$
c=2 \beta_L + 10d^3(\beta_L \vee \beta_N), \quad \hat c=2\lambda_1,\quad k=\frac{\beta_L}2,\quad \hat k= \lambda_1^+. 
$$
\end{theorem}
 
\begin{proof}
Defining $f_p$ and $\tilde f_p$ as in the previous section, we obtain by Lemma \ref{Ableitung}:
$$
f_p(t) \le \exp\Big\{\Big(\lambda_1^+ + \frac{\beta_L}2 p\Big) t + o(t)\Big\}, 
\qquad \tilde f_p(t) \le \exp\Big\{\Big(\lambda_1 + \frac{\beta_L}2 p\Big) t + o(t)\Big\}    
$$
for all $p >0$. Next, we estimate $g_p$ according to formula \eqref{eq:ggleichung}. Observing that $\bar C_2=0$, $\Lambda=\lambda_1$ 
and 
$\sigma=\sqrt{\beta_L}$ (by Lemma \ref{IBFesti}), we obtain 
$$
g_p(t) \le |x-y| \exp\big\{\big(2\lambda_1+2\beta_Lp+ \frac 12 d^3(\beta_L \vee \beta_N) C_p^{2/p} \big)t +o(t) \big\}.
$$
Noting that $C_p^{1/p}\le 2\sqrt{5}p^{1/2} $ for $p \ge 2$ \cite[Proposition 4.2]{BY82}, the assertion follows from 
Theorem \ref{th:zweites}.
\end{proof}

\begin{remark}
Even though we have no reason to believe that the bound in Theorem \ref{iso} is optimal in general, it is at least optimal in case 
$\Delta=0$ (by Lemma  \ref{Ableitung}).
\end{remark}

\section*{Appendix A: A Gronwall-type Lemma}
We include in this section an elementary lemma which we conjecture to be essentially well-known but which we could not 
find in the literature in a version suitable for our needs.
\begin{lemma}\label{le:gronwall}
Let $f:[0,\infty) \to [0,\infty)$ be a locally integrable function and $H:[0,\infty) \to [0,\infty)$ 
non-decreasing such that 
$$
f(t) \le C_1 \int_0^t f(s) \dd s + C_2 \Big( \int_0^t \sqrt{f(s)} \dd s \Big)^2 + H(t)
$$
for some $C_1,C_2\ge 0$ and all $t \ge 0$. Then 
$$
f(t) \le H(t) + H(t)\frac{C_1+\sqrt{C_2}}{C_1+2\sqrt{C_2}} \Big(\exp\{(C_1+2\sqrt{C_2})t\}-1\Big)  
$$
for all $t \ge 0$.
\end{lemma}
\begin{proof} Let $\lambda>0$ and $\gamma_t:=\lambda (1-\ee^{-\lambda t})^{-1}$ for $t>0$. Then, by Jensen's inequality, for $t>0$
\begin{align*}
f(t) &\le C_1 \int_0^t  f(s) \dd s + 
     C_2 \Big(  \int_0^t \sqrt{f(s)} \frac{ \ee^{\lambda (t-s)}}{\gamma_t} \gamma_t \ee^{-\lambda (t-s)} \dd s \Big)^2 + H(t)\\
&\le C_1 \int_0^t  f(s) \dd s + C_2  \int_0^t f(s)  \frac{ \ee^{\lambda (t-s)}}{\gamma_t} \dd s  + H(t).
\end{align*}
Therefore,
\begin{align*}
\ee^{-\lambda t}f(t)&\le C_1 \ee^{-\lambda t} \int_0^t  f(s) \dd s + \frac{C_2}{\lambda}\int_0^t f(s)  
 \ee^{-\lambda s} \dd s  + H(t)  \ee^{-\lambda t}\\
&\le  C_1  \int_0^t   \ee^{-\lambda s}  f(s) \dd s + \frac{C_2}{\lambda}\int_0^t f(s)  
\ee^{-\lambda s} \dd s  + H(t)  \ee^{-\lambda t}.
\end{align*}
Gronwall's Lemma implies
\begin{align*}
f(t) &\le \ee^{\lambda t} \Big(H(t) \ee^{-\lambda t} +  \int_0^t  H(s)  \ee^{-\lambda s}\Big(C_1+\frac{C_2}{\lambda}\Big)
\ee^{\Big(C_1+\frac{C_2}{\lambda}\Big)(t-s)} \dd s \Big)\\
&= H(t)  +  \ee^{\lambda t} \int_0^t  H(s)  \ee^{-\lambda s}\Big(C_1+\frac{C_2}{\lambda}\Big) \ee^{\Big(C_1+\frac{C_2}{\lambda}\Big)(t-s)} \dd s. 
\end{align*}
Choosing $\lambda= \sqrt{C_2}$ (which minimizes $\lambda + C_1 + \frac{C_2}{\lambda}$) and using the monotonicity of $H$, we obtain
$$
f(t) \le H(t) + H(t) \frac {C_1 + \sqrt{C_2}}{C_1 + 2\sqrt{C_2}} \Big( \ee^{(C_1 + 2\sqrt{C_2})t}-1\Big)
$$
as claimed in the lemma.
\end{proof}
\section*{Appendix B: Some Estimates for IBFs}
In this appendix, we collect three basic properties of IBFs which are used in Section \ref{se:IBF} and which 
do not seem to have appeared in the literature so far. 
\begin{lemma}\label{beta} 
Let $\phi$ be an IBF with covariance tensor $b$. Then we obtain for the correlation functions $B_{N/L}$ defined in 
Section \ref{se:IBF}
$$
1-B_L(r) \le \frac{\beta_L}{2} r^2 \mbox{ and } 1-B_N(r) \le \frac{\beta_N}{2} r^2
$$
for all $r \ge 0$.
\end{lemma}
\begin{proof} 
Let $U(x)$, $x \in \R^d$ be an $\R^d$-valued centered Gaussian process with 
$\cov(U_i(x),U_j(y)) = b_{ij}(x-y)$. Denoting the $i^{\mathrm{th}}$ unit coordinate vector by $\ee_i$ and using Schwarz' inequality, 
we get
\begin{align}
B_L''(r)&=\lim_{h \to 0}\lim_{\delta \to 0} \E \Big(\frac{U_1(h\ee_1)-U_1(0)}{h} \frac{U_1(-(r+\delta)\ee_1)-U(-r\ee_1)}{\delta}\Big)\nonumber\\
&=-\E\Big( U_1'(r\ee_1)U_1'(0)\Big) \ge - \E\Big( U_1'(0)^2\Big) = B_L''(0).\label{cov}
\end{align}
Therefore, for each $r>0$ there exists some $\theta \in (0,r)$ such that
$$
B_L(r) = B_L(0) + \frac{1}{2} B_L''(\theta) r^2 \ge 1 + \frac 12 B_L''(0) r^2= 1-\frac 12 \beta_L r^2.
$$
The estimate for $B_N$ follows in the same way, so the assertion of the lemma follows.\\
\end{proof}
Observe that the following lemma holds for every IBF -- even if the top exponent $\lambda_1$ is negative. 
\begin{lemma}\label{IBFesti} Let $\phi$ be an IBF. Let $x,y \in \R^d$, $x \neq y$ 
and let $\rho(t):=|\phi_t(x)-\phi_t(y)|$. Then
$$
\E \rho_t^q \le |x-y|^q \exp\Big\{ \Big(q\lambda_1 + q^2\frac{\beta_L}{2}\Big) t\Big\}
$$
for all $t\ge 0$ and all $q \ge 1$.
\end{lemma}
\begin{proof} 
We know that
$$
\dd \rho_t = (d-1) \frac{1-B_N(\rho_t)}{\rho_t} \dd t + \sqrt{2(1-B_L(\rho_t))} \dd W_t.
$$
Let $q \ge 1$. It\^o's formula implies
\begin{align*}
\dd \rho_t^q &= q\rho_t^{q-1}\dd \rho_t + \frac{q (q-1)}{2} \rho_t^{q-2} \dd \langle \rho \rangle_t \\
&=q  \rho_t^{q-2} \big( (d-1)(1-B_N(\rho_t))+(q-1)(1-B_L(\rho_t))\big) \dd t
+ q \rho_t^{q-1} \sqrt{2(1-B_L(\rho_t))} \dd W_t.
\end{align*}
By Lemma \ref{beta}, we get
$$
\E \rho_t^q \le  |x-y|^q + q \big(\frac{\beta_N}{2}(d-1) + \frac{\beta_L}{2}(q-1)\big) \int_0^t  \E \rho_s^q \dd s.
$$
Gronwall's Lemma, together with the formula for the top exponent $\lambda_1=(d-1)\beta_N/2-\beta_L/2$ imply the assertion.
\end{proof}
Next, we look at the derivative of an IBF. The formula in the next lemma becomes particularly nice if we use the following 
{\em Schatten norm} of a  $d\times d$-matrix $A=(a_{ij})$:
$$
\| A \|_S := \left( \sum_{i,j=1}^{d} \left(  \sum_{k=1}^d a_{ik} a_{jk} \right)^2 \right)^{1/4}
= \left(\sum_{i=1}^d \sigma_i^4\right)^{1/4},
$$
where $\sigma_1,...,\sigma_d$ are the singular values of $A$.
\begin{lemma}\label{Ableitung}
 Let $\phi$ be an IBF. Then, for each $x \in \R^d$, there exists a one-dimensional Wiener process $W$
such that for all $t \ge 0$, 
$$
\|D\phi_t(x)\|_S = d^{1/4} \exp\{ \lambda_1t + \sqrt{\beta_L}W_t\}.
$$
\end{lemma}
\begin{proof}
Fix $x \in \R^d$ and define
$$
N_t:=\int_0^t \sum_{i,j,k} D_k M^i(\phi_s(x),\dd s) \frac{D_j\phi_s^j(x) D_j\phi_s^k(x) }{\|D\phi_s(x)\|_S}.
$$
Then it is easy to see (cf.~\cite[p. 101]{HvB10}) that
$$
\langle N \rangle_t = \beta_L t.
$$
Therefore, $W_t:=(\beta_L)^{-1/2} N_t$, $t \ge 0$ is a standard Brownian motion and applying It\^o's formula, we get
$$
\log\|D\phi_t(x)\|_S= \frac 14 \log d + N_t + \lambda_1 t =   \frac 14 \log d + \sqrt{\beta_L} W_t + \lambda_1 t.
$$
Exponentiating this expression, the lemma follows.
\end{proof}

\bibliographystyle{abbrv}


\end{document}